\documentclass{amsart}
\usepackage{amssymb,amscd,xypic}
\input{xypic}

\unitlength=1mm \linethickness{0.5pt} 

\newtheorem{theorem}{Theorem}[section]
\newtheorem{proposition}[theorem]{Proposition}
\newtheorem{corollary}[theorem]{Corollary}
\newtheorem{lemma}[theorem]{Lemma}
\newtheorem{question}[theorem]{Question}
\theoremstyle{definition}

\newtheorem{example}[theorem]{Example}

\theoremstyle{remark}
\newtheorem*{remark}{Remark}

\numberwithin{equation}{section}
\def\C{\mathbb C}

\def\Q{\mathbb Q}
\def\R{\mathbb R}

\def\Z{\mathbb Z}

\def\k{\mathbf k}

\def\phi{\varphi}

\newcommand{\mb}[1]{{\textbf {\textit#1}}}

\newcommand{\bin}[2]{{\textstyle\binom{#1}{#2}}}

\renewcommand{\ge}{\geqslant}
\renewcommand{\le}{\leqslant}

\def\bideg{\mathop{\mathrm{bideg}}}
\def\Ext{\mathop{\mathrm{Ext}}\nolimits}

\def\Tor{\mathop{\mathrm{Tor}}\nolimits}

\newcommand{\cs}{\mathbin{\#}}

\newcommand{\djs}{\mbox{$D{\mskip-1mu}J\/$}}
\newcommand{\zk}{\mathcal Z_{K}}
\newcommand{\barzk}{\mathcal Z_{\overline{K}}}
\newcommand{\z}{\mathcal{Z}}

\newcommand{\namedright}[3]{\ensuremath{#1\stackrel{#2}
 {\longrightarrow}#3}}
\newcommand{\nameddright}[5]{\ensuremath{#1\stackrel{#2}
 {\longrightarrow}#3\stackrel{#4}{\longrightarrow}#5}}
\newcommand{\namedddright}[7]{\ensuremath{#1\stackrel{#2}
 {\longrightarrow}#3\stackrel{#4}{\longrightarrow}#5
  \stackrel{#6}{\longrightarrow}#7}}

\newcommand{\larrow}{\relbar\!\!\relbar\!\!\rightarrow}
\newcommand{\llarrow}{\relbar\!\!\relbar\!\!\larrow}

\newcommand{\lnamedright}[3]{\ensuremath{#1\stackrel{#2}
 {\larrow}#3}}

\newcommand{\llnamedright}[3]{\ensuremath{#1\stackrel{#2}
 {\llarrow}#3}}
\newcommand{\llnameddright}[5]{\ensuremath{#1\stackrel{#2}
 {\llarrow}#3\stackrel{#4}{\llarrow}#5}}

\begin{document}

\title[Homotopy types of moment-angle complexes]%
{The homotopy types of moment-angle complexes for flag complexes}

\author{Jelena Grbi\' c$^{*}$}
\address{School of Mathematics, University of Southampton, Southampton SO17 1BJ,
UK} \email{J.Grbic@soton.ac.uk}

\author{Taras Panov$^{**}$}
\address{Department of Mathematics and Mechanics, Moscow
State University, Leninskie Gory, 119991 Moscow, Russia,
\newline\indent Institute for Theoretical and Experimental Physics,
Moscow, Russia \quad \emph{and}
\newline\indent Institute for Information Transmission Problems,
Russian Academy of Sciences}
%Taras
%\newline\indent Delone Laboratory of Discrete and
%Computational Geometry, Yaroslavl State University, Yaroslavl, Russia}
\email{tpanov@mech.math.msu.su}

\author{Stephen Theriault}
\address{School of Mathematics, University of Southampton, Southampton
SO17 1BJ, UK} \email{S.D.Theriault@soton.ac.uk}

\author{Jie Wu$^{***}$}
\address{Department of Mathematics, National University of
Singapore, Block S17 (SOC1), 06-02 10, Lower Kent Ridge Road,
119076 Singapore} \email{matwuj@nus.edu.sg}

%Taras 
\thanks{$^{*}$ Research supported by the Leverhulme Trust (Research 
Project Grant RPG-2012-560).}  
\thanks{$^{**}$ Research supported by the Russian Science Foundation (grant
no.~14-11-00414).}
\thanks{$^{***}$ Research supported in part by the Singapore Ministry
of Education research grant (AcRF Tier 1 WBS No.
R-146-000-190-112) and a grant (No.~11329101) of NSFC of China.}

%\keywords{moment-angle manifold, simplicial fan, simple polytope,
%complex structure, Dolbeault cohomology, Hodge numbers}

%\subjclass{32J18, 32L05, 32Q55, 57R19, 14M25}

%\date{\today}

\begin{abstract}
We study the homotopy types of moment-angle complexes, or
equivalently, of complements of coordinate subspace arrangements.
The overall aim is to identify the simplicial complexes $K$ for
which the corresponding moment-angle complex $\zk$ has the
homotopy type of a wedge of spheres or a connected sum of sphere
products. When $K$ is flag, we identify in algebraic and
combinatorial terms those $K$ for which $\zk$ is homotopy
equivalent to a wedge of spheres, and give a combinatorial formula
for the number of spheres in the wedge. This extends results of
Berglund and J\"ollenbeck on Golod rings and homotopy theoretical
results of the first and third authors. We also establish a
connection between minimally non-Golod rings and moment-angle
complexes $\zk$ which are homotopy equivalent to a connected sum
of sphere products. We go on to show that for any flag complex $K$
the loop spaces $\Omega \zk$ and $\Omega \djs(K)$ are homotopy
equivalent to a product of spheres and loops on spheres when
localised rationally or at any prime $p\neq 2$.
\end{abstract}

\maketitle

\section{Introduction}
Moment-angle complexes are key players in the emerging field of
\emph{toric topology}, which lies on the borders between topology,
algebraic and symplectic geometry, and
combinatorics~\cite{bu-pa12}. The moment-angle complex $\zk$, as a
space with a torus action, appeared in work of Davis and
Januszkiewicz~\cite{da-ja91} on topological generalisations of
toric varieties. The homotopy orbit space of $\zk$ is the
Davis--Januszkiewicz space $\djs(K)$, which is a cellular model
for the Stanley--Reisner ring $\mathbb Z[K]$, while the genuine
orbit space of $\zk$ is the cone over the simplicial complex $K$.
Buchstaber and the second author~\cite{bu-pa02} introduced
homotopy theoretical models of both the moment-angle complex $\zk$
and the Davis--Januszkiewicz space $\djs(K)$ as a homotopy colimit
construction of the product functor on the topological pairs
$(D^2, S^1)$ and $(\mathbb CP^\infty, \ast)$ respectively, with
the colimit taken over the face category of the simplicial complex
$K$. Recently, homotopy theoretical generalisations of
moment-angle complexes and related spaces under the unifying
umbrella of polyhedral products~(see, for
example,\cite{b-b-c-g10},\cite{gr-th04}, \cite{gr-th07},
\cite{gr-th12}) have brought stable and unstable decomposition
techniques to bear, and are leading to an improved understanding
of toric spaces.

The homotopy theory of moment-angle complexes and polyhedral
products in general has far reaching applications in
combinatorial and homological algebra, in particular, in the study of \emph{face rings} (or
\emph{Stanley--Reisner rings}) of simplicial complexes and more
general monomial ideals.

In this paper we consider the following
related homotopy theoretical and algebraic problems:
\begin{itemize}
\item[--] identifying the homotopy type of
the moment-angle complex~$\zk$ for certain simplicial complexes
$K$;
\item[--] describing the multiplication and higher Massey
products in the $\Tor$-algebra
$H^*(\zk)=\Tor_{\k[v_1,\ldots,v_m]}(\k[K],\k)$ of the face
ring~$\k[K]$;
\item[--] describing the Yoneda algebra $\Ext_{\k[K]}(\k,\k)$ in terms of
generators and relations;
\item[--] describing the structure of the Pontryagin algebra
$H_*( \Omega\djs(K))$ and its commutator subalgebra
$H_*( \Omega\zk)$ via iterated and higher Whitehead (Samelson)
products;
\item[--] identifying the homotopy type of the loop spaces
$ \Omega\djs(K)$ and $ \Omega\zk$.
\end{itemize}

The main objects and constructions are introduced in Section~\ref{preli},
together with some known preliminary results. In Section~\ref{sec:golod} we
give topological interpretations of the Golod property of the
face ring~$\k[K]$. This ring is \emph{Golod} if the multiplication in the
$\Tor$-algebra $H^*(\zk)=\Tor_{\k[v_1,\ldots,v_m]}(\k[K],\k)$ is
trivial, together with all higher Massey products (cf.~\cite{golo62},
\cite{gu-le69}). The topological interpretations are in terms of
$H_{\ast}(\Omega\zk)$ being a free graded associative algebra,
$H^{\ast}(\zk)$ having a trivial multiplication, and a certain
identity holding for the Poincar\'{e} series of $H_{\ast}(\Omega\zk)$.

In Section~\ref{flagK} we concentrate on the case when $K$ is a
flag complex. Our techniques allow for a complete solution of the
problems above in the case of flag complexes.  A flag complex $K$ is
determined by its 1-skeleton $K^1$. The Yoneda algebra
$\Ext_{\k[K]}(\k,\k)\cong H_*( \Omega\djs(K))$ has a simple
presentation as a graph product algebra. In Theorem~\ref{multgen}
we explicitly describe the minimal generating set of its commutator subalgebra
$H_*( \Omega\zk)$ and the basis of the corresponding iterated
commutators.

From the homotopy-theoretic point of view, particularly important
moment-angle complexes $\zk$ are those which have the homotopy type of
a wedge of spheres. In this case the associative graded algebra
$H_*( \Omega\zk)$ is free, and the multiplication in the
$\Tor$-algebra $H^*(\zk)=\Tor_{\k[v_1,\ldots,v_m]}(\k[K],\k)$ is
trivial, together with all higher Massey products, so the face
ring~$\k[K]$ is Golod. In Theorem~\ref{flws}
we show that for flag complexes~$K$ the Golodness of $K$ is
the precise algebraic criterion for $\zk$ being homotopy
equivalent to a wedge of spheres. Using a result of Berglund and
J\"ollenbeck~\cite{be-jo07}, this can be reformulated entirely in
terms of the cup product:  for a flag complex $K$, the moment-angle
complex $\zk$ is homotopy equivalent to a wedge of spheres if and only
if the cup product in $H^*(\zk)$ is trivial. Most importantly,
there is a purely combinatorial description of the class of flag
complexes $K$ for which $\zk$ is homotopy equivalent to a wedge of spheres: the
1-skeleton of such $K$ must be a \emph{chordal graph}. This is
an important concept in applied combinatorics and optimisation;
the vertices in a chordal graph admit a
%Taras
\emph{perfect elimination ordering}~\cite{fu-gr65}.
%\emph{total elimination ordering}~\cite{fu-gr65}.

For general $K$, the Golod property of $\k[K]$ does do not
guarantee that $\zk$ is homotopy equivalent to a wedge of spheres.
The reason is that for some Golod complexes $K$, the cohomology
ring $H^*(\zk;\Z)$ may contain non-trivial torsion (see
Example~\ref{rp2ex}). Especially intriguing is that for all known
examples of Golod complexes $K$, the moment-angle complex $\zk$ is
a co-$H$-space (and even a suspension), and this may as well be
true in general (see Question~\ref{GcoHs}).

The next homotopy type of $\zk$ which we consider is a connected
sum of sphere products, where each summand is a product of exactly two spheres. Such a
$\zk$ is obtained by attaching a top
%Taras
cell
to a wedge of spheres along one commutator relation. The
corresponding face ring $\k[K]$ is \emph{minimally non-Golod}, and
the commutator subalgebra $H_*( \Omega\zk)$ in the Yoneda algebra
$\Ext_{\k[K]}(\k,\k)\cong H_*( \Omega\djs(K))$ is a one-relator
algebra. In the case of a flag simplicial complex $K$ the previous
statement classifies minimally non-Golod Stanley-Reisner rings
$\k[K]$, that is, $\k[K]$ is minimally non-Golod if and only if
the moment-angle complex $\zk$ is homotopy equivalent to a
connected sum of sphere products. It is an open question whether
this classification criteria holds for a general simplicial
complex (see Question~\ref{minnG}).

In Section~\ref{sec:loopzk} we address the last problem in the
list above. Our main result there is Theorem~\ref{Whflag}, which
shows that for a flag $K$, both $ \Omega\zk$ and $ \Omega\djs(K)$
are homotopy equivalent to products of spheres and loops of
spheres when localised rationally or at any prime $p\neq 2$.
We also show that the integral
Pontryagin algebra $H_*( \Omega\zk)$ is torsion-free
(Corollary~\ref{zktorsionfree}).

In Section~\ref{pentex} we give a detailed illustration of
many of the ideas and results of the paper in the case when $K$
is the boundary of a pentagon.

\medskip

The authors would like to thank the International Centre for
Mathematical Sciences in Edinburgh, whose support
through a Research-in-Groups grant allowed the authors
to work together on these problems for a month in Aberdeen.

\section{Preliminaries}\label{preli}
Let $K$ be a finite simplicial complex on the set
$[m]=\{1,2,\ldots,m\}$, that is, a collection of subsets
$I=\{i_1,\ldots,i_k\}\subset[m]$ closed under inclusion. We refer
to $I\in K$ as \emph{simplices} or \emph{faces} of~$K$, and
always assume that $\varnothing\in K$.

Assume we are given a set of $m$ topological pairs
\[
  (\mb X,\mb A)=\{(X_1,A_1),\ldots,(X_m,A_m)\}
\]
where $A_i\subset X_i$. For each simplex $I\in K$ we set
\[
  (\mb X,\mb A)^I=\bigl\{(x_1,\ldots,x_m)\in
  \prod_{i=1}^m X_i\ | \  x_i\in A_i\quad\text{for }i\notin
  I\bigl\}.
\]
The \emph{polyhedral product} of $(\mb X,\mb A)$ corresponding to
$K$ is the following subset in~$\prod_{i=1}^m X_i$
\[
  (\mb X,\mb A)^K=\bigcup_{I\in K}(\mb X,\mb A)^I=
  \bigcup_{I\in K}
  \Bigl(\prod_{i\in I}X_i\times\prod_{i\notin I}A_i\Bigl).
\]
In the case when all the pairs $(X_i,A_i)$ are the same, that is,
$X_i=X$ and $A_i=A$ for $i=1,\ldots,m$, we use the notation
$(X,A)^K$ for $(\mb X,\mb A)^K$.

The main example of the polyhedral product is the
\emph{moment-angle complex} $\zk=(D^2,S^1)^K$~\cite{bu-pa00},
which is the key object of study in toric topology. The space
$\zk$ has a natural coordinatewise action of the torus~$T^m$, and
it is a manifold whenever $K$ is a triangulation of a sphere.
Other important cases of polyhedral products include $\djs(K)=(\C
P^\infty,\ast)^K$, which is referred to as the
\emph{Stanley--Reisner space}~\cite{bu-pa00} or the
\emph{Davis--Januszkiewicz space}~\cite{p-r-v04}, and the
complement of the complex \emph{coordinate subspace arrangement}
corresponding to~$K$
\[
  U(K)=(\C,\C^*)^K=\C^m\setminus\bigcup_{\{i_1,\ldots,i_k\}\notin K}
  \{z_{i_1}=\cdots=z_{i_k}=0\}.
\]

According to~\cite[Th.~5.2.5]{bu-pa00}, there is a
$T^m$-equivariant deformation retraction $U(K)\to\zk$. The
spaces $\zk$ and $(\C P^\infty,\ast)^K$ are related by the following
result.

\begin{proposition}[{\cite[Cor.~3.4.5]{bu-pa00}}]
There is a homotopy fibration
\[
  \zk\longrightarrow\djs(K)\longrightarrow(\C
  P^\infty)^m
\]
that is, $\zk$ is the homotopy fibre of the canonical inclusion
$\djs(K)\to(\C P^\infty)^m$.
\end{proposition}

This fibration splits after looping
\[
   \Omega
  \djs(K)\simeq \Omega\zk\times T^m
\]
but this is not an $H$-space splitting. One can think of $
\Omega\zk$ as the ``commutator subgroup'' of $ \Omega\djs(K)$,
although this can be made precise only after passing to Pontryagin
(loop homology) algebras.

\begin{proposition}[{\cite[(8.2)]{pa-ra08}}]
There is an exact sequence of (noncommutative) algebras
\begin{equation}\label{commutator}
  1\longrightarrow H_*( \Omega\zk;\k)\longrightarrow
  H_*( \Omega\djs(K);\k)\longrightarrow
  \Lambda[u_1,\ldots,u_m]\longrightarrow1
\end{equation}
where $\k$ is field or~$\Z$, and $\Lambda[u_1,\ldots,u_m]$ is the
exterior algebra on $m$ generators of degree one.
\end{proposition}

In what follows we shall often omit the coefficient ring $\k$ in
the notation of (co)homology.

The exterior algebra $\Lambda[u_1,\ldots,u_m]$ can be thought of
as the abelianisation of a largely noncommutative algebra
$H_*( \Omega\djs(K))$ (we expand on this below), so that
$H_*( \Omega\zk)$ is its commutator subalgebra.

The \emph{face ring} of $K$ (also known as the
\emph{Stanley--Reisner ring}) is defined as the quotient of the
polynomial algebra $\k[v_1,\ldots,v_m]$ by the square-free
monomial ideal generated by non-simplices of~$K$
\[
  \k[K]=\k[v_1,\ldots,v_m]\big/\bigl(v_{i_1}\cdots v_{i_k}\ |\
  \{i_1,\ldots,i_k\}\notin K\bigr).
\]
We make it graded by setting $\deg v_i=2$.

\begin{theorem}[{\cite{da-ja91}, \cite[Prop.~3.4.3]{bu-pa00}}]
There is an isomorphism of graded commutative algebras
\[
  H^*\bigl(\djs(K);\k\bigr)\cong\k[K]
\]
for any coefficient ring $\k$.
\end{theorem}

The cohomology ring $H^*(\zk;\k)$ and the Pontryagin algebra $H_*(
\Omega\djs(K);\k)$ decode different homological invariants of the
face ring~$\k[K]$, as is stated next.

\begin{theorem}[{\cite[Th.~5.3.4]{bu-pa00}}]\label{odjcoh}
If $\k$ is a field, then there is an isomorphism of graded
noncommutative algebras
\[
  H_*\bigl( \Omega\djs(K);\k\bigr)\cong\Ext_{\k[K]}(\k,\k)
\]
where $\Ext_{\k[K]}(\k,\k)$ is the Yoneda algebra of~$\k[K]$.
\end{theorem}

This is proved by applying the Adams cobar spectral sequence to the
loop fibration $ \Omega\djs(K)\to\mathcal P\djs(K)\to\djs(K)$, where $\mathcal P\djs(K)$ is the space of based paths in $\djs(K)$ and using
the formality of $\djs(K)$.

\begin{theorem}[\cite{bu-pa00}, \cite{b-b-p04},
\cite{fran06}]\label{zkcoh}
If $\k$ is a field or $\Z$, then there
are isomorphisms of (bi)graded commutative algebras
\begin{align*}
  H^*(\zk)&\cong\Tor_{\k[v_1,\ldots,v_m]}\bigl(\k[K],\k\bigr)\\
  &\cong H\bigl[\Lambda[u_1,\ldots,u_m]\otimes\k[K],d\bigr]\\
  &\cong \bigoplus_{I\subset[m]}\widetilde H^*(K_I).
\end{align*}
Here, the second row is the cohomology of the differential
bigraded algebra with $\bideg u_i=(-1,2)$, $\bideg v_i=(0,2)$ and
$du_i=v_i$, $dv_i=0$ (the Koszul complex). In the third row,
$\widetilde H^*(K_I)$ denotes the reduced simplicial cohomology
of the full subcomplex $K_I\subset K$ (the restriction of $K$
to $I\subset[m]$). The last isomorphism is the sum of isomorphisms
\[
  H^p(\zk)\cong
  \sum_{I\subset[m]}\widetilde H^{p-|I|-1}(K_I),
\]
and the ring structure (the Hochster ring) is given by the maps
\[
  H^{p-|I|-1}(K_I)\otimes H^{q-|J|-1}(K_J)\to
  H^{p+q-|I|-|J|-1}(K_{I\cup J})
\]
which are induced by the canonical simplicial maps $K_{I\cup
J}\to K_I*K_J$ for $I\cap J=\varnothing$ and zero otherwise.
\end{theorem}

In~\cite{gr-th07} several classes of complexes~$K$ have been
identified for which $\zk$ has homotopy type of a wedge of
spheres. These include all skeleta of simplices, and the so-called
\emph{shifted} complexes. One special case which we shall refer to
several times later is when $K$ is a disjoint union of finitely many vertices.

\begin{theorem}[{\cite{gr-th04}}]\label{disjointpoints}
   Let $K$ be the disjoint union of $m$ points. Then there is a
   homotopy equivalence
   \[\zk\simeq\bigvee_{\ell=2}^{m} (S^{\ell+1})^{\vee(\ell-1)\binom{m}{\ell}}.\]
\end{theorem}

Further, in~\cite{gr-th07} it was shown that there is a way to build
new complexes $K$ whose corresponding $\zk$ is a wedge of spheres
from existing ones.

\begin{theorem}[{\cite[Th.~10.1]{gr-th07}}]
Assume that $\mathcal Z_{K_1}$ and $\mathcal Z_{K_2}$ both
have homotopy type of a wedge of spheres, and $K$ is obtained by
attaching $K_1$ to $K_2$ along a common face. Then $\zk$ also
has homotopy type of a wedge of spheres.
\end{theorem}

\begin{corollary}\label{orderws}
Assume that there is an order $I_1,\ldots,I_s$ of the maximal
faces of~$K$ such that $\bigr(\bigcup_{j<k}I_j\bigl)\cap I_k$ is
a single face for each $k=1,\ldots,s$. Then $\zk$ has homotopy
type of a wedge of spheres.
\end{corollary}

%The property of $\zk$ being a wedge of spheres reflects
%algebraically in both $H^*(\zk)$ and $H_*( \Omega\djs(K))$:
%the former has trivial multiplication and the latter is a free
%algebra; furthermore, the Poincar\'e series of both algebras can
%be expressed in terms of each other. (The \emph{Poincar\'e series}
%of a graded $\k$-vector space $A=\bigoplus_{i\ge0}A^i$ is given by
%$P(A;t)=\sum_{i\ge0}\dim A^i$.)

\section{The Golod property}
\label{sec:golod}

In this section we give topological interpretations of the Golod
property. The face ring $\k[K]$ is called \emph{Golod}
(cf.~\cite{gu-le69}) if the multiplication and all higher Massey
operations in $\Tor_{\k[v_1,\ldots,v_m]}\bigl(\k[K],\k\bigr)$ are
trivial. The Golod property can be defined for general graded or
local Noetherian rings. Several combinatorial criteria for
Golodness were given in~\cite{h-r-w99}. We say that the simplicial
complex $K$ is Golod if $\k[K]$ is a Golod ring. In view of
Theorem~\ref{zkcoh}, the Golod property is an algebraic
approximation to the property of $\zk$ being homotopy equivalent
to a wedge of spheres, although this approximation is not exact as
Example~\ref{rp2ex} below shows. By a result of Berglund and
J\"ollenbeck~\cite[Th.~5.1]{be-jo07}, $K$ is a Golod complex if
the multiplication in
$\Tor_{\k[v_1,\ldots,v_m]}\bigl(\k[K],\k\bigr)$ is trivial, i.e.
there is no need to check the triviality of higher Massey products
in the case of face rings.

Our main result in this section is Theorem~\ref{golodcond}, but before
stating this we give a more general result which is of independent interest.
Recall that the \emph{Poincar\'e series} of a graded $\k$-module
$A=\bigoplus_{i\ge0}A^i$ is given by $P(A;t)=\sum_{i\ge0}\dim A^i$.

\begin{proposition}
\label{trivmult}
Let $X$ be a simply-connected CW-complex such that
$H_*( \Omega X;\k)$ is a graded free associative algebra, where $\k$ is a
field. Then $H^*(X;\k)$ has trivial multiplication.
\end{proposition}
\begin{proof}
Let $
  Q=H_{>0}( \Omega X)\big/\bigl(H_{>0}( \Omega X)\cdot
  H_{>0}( \Omega X)\bigl)
$ be the space of indecomposable elements, so that $H_*( \Omega
X)=T\langle Q\rangle$ by assumption, where $T\langle Q\rangle$
denotes the free associative algebra on the graded
$\k$-module~$Q$.

Consider the Rothenberg--Steenrod (bar) spectral sequence, which
has $E_2$-term $E_2^{\mathrm{b}}=\Tor_{H_*( \Omega
X)}(\k,\k)$ and converges to $H_*(X)$. By assumption,
\[
  E^{\mathrm{b}}_2=\Tor_{T\langle Q\rangle}(\k,\k)\cong\k\oplus Q
\]
as a $\k$-module. We therefore obtain the following inequalities
for the Poincar\'e series:
\begin{equation}\label{bariq}
P\bigl(\varSigma^{-1}\widetilde H_*(X);t\bigr)=
P(E^{\mathrm{b}}_\infty;t)-1\le P(E^{\mathrm{b}}_2;t)-1=P(Q;t).
\end{equation}

Now consider the Adams (cobar) spectral sequence, which has $E_2$-term
$E^{\mathrm{c}}_2=\mathop{\mathrm{Cotor}}_{H_*(X)}(\k,\k)$ and
converges to $H_*( \Omega X)$. We have a series of inequalities:
\[
  P\bigl(H_*( \Omega X);t\bigr)
  =P(E^{\mathrm{c}}_\infty;t)\le P(E^{\mathrm{c}}_2;t)\le
  P\bigl(T\langle\varSigma^{-1}\widetilde H_*(X)\rangle;t\bigr)\le
  P\bigl(T\langle Q\rangle;t\bigr),
\]
where the second-to-last inequality follows from the cobar
construction (it turns to equality when all differentials in the
cobar construction on $H_*(X)$ are trivial), and the last
inequality follows from~\eqref{bariq}. Now, $P\bigl(H_*( \Omega
X);t\bigr)=P\bigl(T\langle Q\rangle;t\bigr)$ by assumption, so all
inequalities above turn into equalities, and both spectral
sequences collapse at the $E_2$-term. It follows from the collapse
of both spectral sequences that the homology map
\[
  \widetilde H_*(\varSigma \Omega X)=
  \varSigma\widetilde H_{*}( \Omega X)\to\widetilde H_*(X)
\]
induced by the evaluation $\varSigma \Omega X\to X$ is onto.
Consider the commutative diagram
\[
\xymatrix{
  \widetilde H_*(\varSigma \Omega X) \ar[r] \ar[d]_{\Delta} &
  \widetilde H_*(X)\ar[d]^\Delta\\
  \widetilde H_*(\varSigma \Omega X)\otimes
  \widetilde H_*(\varSigma \Omega X) \ar[r]
  & \widetilde H_*(X)\otimes\widetilde H_*(X)
}
\]
in which the vertical arrows are comultiplications, and the
horizontal ones are surjective. Since $\varSigma \Omega X$ is a
suspension, the left arrow is zero, hence, the right arrow is also
zero. By duality, the multiplication in $H^*(X)$ is trivial.
\end{proof}

The Golod property of $K$ has the following topological
interpretations.

\begin{theorem}\label{golodcond}
Let $\k$ be a field. The following conditions are equivalent:
\begin{itemize}
\item[(a)] $H_*( \Omega\zk)$ is a graded free
associative algebra;
\item[(b)] the multiplication in $H^*(\zk)$ is trivial;
\item[(c)] there is the following identity for the Poincar\'e
series:
\[
  P\bigl(H_*( \Omega\zk);t\bigr)=
  \frac1{1-P\bigl(\varSigma^{-1}\widetilde
  H{\mathstrut}^*(\zk);t\bigr)},
\]
where $\varSigma^{-1}$ denotes the desuspension of a graded
$\k$-module.
\end{itemize}
\end{theorem}
\begin{proof}
The implication (a)$\Rightarrow$(b) holds by Proposition~\ref{trivmult}.

To prove the implication (b)$\Rightarrow$(c) we use the above
mentioned result~\cite[Th.~5.1]{be-jo07}, according to which if
the product in
$H^*(\zk)=\Tor_{\k[v_1,\ldots,v_m]}\bigl(\k[K],\k\bigr)$ is
trivial, then all higher Massey operations are also trivial, that
is, $\k[K]$ is Golod. By the alternative definition of the Golod
property~\cite{gu-le69}, $\k[K]$ is Golod if and only if the
following identity for the Poincar\'e series holds:
\[
  P\bigl(\Ext_{\k[K]}(\k,\k);t\bigr)=\frac{(1+t)^m}
  {1-\sum_{i,j>0}\dim
  \Tor^{-i,2j}_{\k[v_1,\ldots,v_m]}\bigl(\k[K],\k\bigr)t^{-i+2j-1}}.
\]
Using Theorems~\ref{odjcoh} and~\ref{zkcoh}, we rewrite this as
\[
  P\bigl(H_*( \Omega\djs(K));t\bigr)=\frac{P\bigl(H_*(T^m);t\bigr)}
  {1-P\bigl(\varSigma^{-1}\widetilde H{\mathstrut}^*(\zk);t\bigr)}.
\]
Since $ \Omega\djs(K)\simeq \Omega\zk\times T^m$, the above
identity is equivalent to that of~(c).

To prove the implication (c)$\Rightarrow$(a) we observe that
\[
  \frac1{1-P\bigl(\varSigma^{-1}
  \widetilde H{\mathstrut}^*(\zk);t\bigr)}
  =P\bigl(T\langle\varSigma^{-1}\widetilde H_*(\zk)\rangle;t\bigr),
\]
so the identity from (c) is equivalent to
$P\bigl(H_*( \Omega\zk);t\bigr)=P\bigl(T\langle\varSigma^{-1}\widetilde
H_*(\zk)\rangle\bigr)$. Hence, all differentials in the cobar
construction on~$H_*(\zk)$ are trivial, which implies that
$H_*( \Omega\zk)$ is a free associative algebra on
$\varSigma^{-1}\widetilde H_*(\zk)$.
\end{proof}

The conditions of Theorem~\ref{golodcond} do not guarantee that
$\zk$ is homotopy equivalent to a wedge of spheres. One reason is
that $H^*(\zk;\Z)$ may contain arbitrary torsion. This follows
easily from Theorem~\ref{zkcoh}: since $\widetilde H^*(K)$ is a
direct summand in $H^*(\zk)$, one may take $K$ to be a
triangulation of a space with torsion in cohomology. The simplest
example is the 6-vertex triangulation of $\R P^2$.

\begin{example}\label{rp2ex}
Let $K$ be the simplicial complex shown in Fig.~\ref{rptri},
where the vertices with the same labels are identified, and the
boundary edges are identified according to the orientation shown.
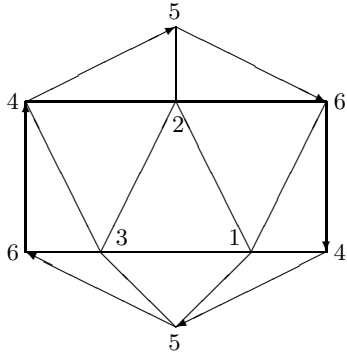
\begin{figure}[h]
  \begin{center}
  \begin{picture}(120,40)
  \put(59,-3){\small 5}
  \put(59,41){\small 5}
  \put(81,29){\small 6}
  \put(81,9){\small 4}
  \put(37.5,29){\small 4}
  \put(37.5,9){\small 6}
  \put(59.5,26){\small 2}
  \put(52,11){\small 3}
  \put(67,11){\small 1}
  \put(60,0){\vector(-2,1){20}}
  \put(40,10){\vector(0,1){20}}
  \put(40,30){\vector(2,1){20}}
  \put(60,40){\vector(2,-1){20}}
  \put(80,30){\vector(0,-1){20}}
  \put(80,10){\vector(-2,-1){20}}
  \put(40,10){\line(1,0){40}}
  \put(40,30){\line(1,0){40}}
  \put(60,0){\line(-1,1){10}}
  \put(60,0){\line(1,1){10}}
  \put(50,10){\line(-1,2){10}}
  \put(50,10){\line(1,2){10}}
  \put(70,10){\line(-1,2){10}}
  \put(70,10){\line(1,2){10}}
  \put(60,30){\line(0,1){10}}
  \end{picture}
  \end{center}
  \caption{6-vertex triangulation of $\R P^2$.}
  \label{rptri}
\end{figure}
A calculation using Theorem~\ref{zkcoh} shows that the nontrivial
cohomology groups of $\zk$ are given by
\[
  H^0=\Z,\quad H^5=\Z^{10},\quad H^6=\Z^{15},\quad H^7=\Z^6,
  \quad H^9=\Z/2.
\]
Therefore, all products and Massey products vanish for dimensional
reasons, so $K$ is Golod (over any field). Nevertheless, $\zk$
is not homotopy equivalent to a wedge of spheres because of the
torsion. In particular, in this example we have
\begin{equation}
  \label{rp2example}
  \zk\simeq (S^5)^{\vee 10}\vee(S^6)^{\vee 15}\vee(S^7)^{\vee 6}
  \vee\varSigma^7\R P^2
\end{equation}
where $X^{\vee k}$ denotes the $k$-fold wedge of~$X$. For, if we
regard $\zk$ as a $CW$-complex built up by attaching $k$-cells to
the $(k-1)$-skeleton for $6\le k\le 9$, then the attaching maps
are all in the stable range. But stably these attaching maps are
all null homotopic since, by~\cite{b-b-c-g10}, the homotopy
equivalence in~(\ref{rp2example}) holds after one suspension.
Therefore the attaching maps are null homotopic, and
so~(\ref{rp2example}) holds without having to suspend.
\end{example}

\begin{question}\label{GcoHs}
Assume that $H^*(\zk)$ has trivial multiplication, so that $K$
is Golod, over any field. Is it true that $\zk$ is a co-$H$-space,
or even a suspension, as in all known examples?
\end{question}

Denote by $K_{\widehat i}$ the restriction of $K$ to the set of vertices $[m]\setminus\{i\}$, that is,
%Taras
$K_{\widehat i}=\{J\in K\ |\ i\notin J\}$.
%$K_{\widehat i}=\{\tau\in K\ |\ i\notin\tau\}.$
 It follows from the description of the product in
$H^*(\zk)$ in Theorem~\ref{zkcoh} that if $K$ is Golod, then
$K_{\widehat i}$ is also Golod. Following~\cite{be-jo07}, we
refer to $K$ as a \emph{minimally non-Golod complex} if $K$ is
not Golod, but $K_{\widehat i}$ is Golod for each~$i$.

The condition for $K$ to be minimally non-Golod is an ``algebraic
approximation" of the topological condition for $\zk$ to be
homeomorphic to a connected sum of sphere products, with two
spheres in each product. In what follows, whenever we say that
$\zk$ is a connected sum of sphere products, we mean that each
summand is a product of exactly two spheres. (In fact, there is no known
example of $\zk$ which is homeomorphic to a nontrivial connected
sum of sphere products with more than two spheres in at least one
product.)

To justify the term ``algebraic approximation'', the following question
needs to be positively answered.

\begin{question}\label{minnG}
Is it true that if $\zk$ is a connected sum of sphere products,
then $K$ is minimally non-Golod?
\end{question}

Examples of minimally non-Golod complexes include the boundary
complexes of polygons and, more generally, stacked polytopes
different from simplices~\cite[Th.~6.19]{be-jo07}. For all these
cases it is known that $\zk$ is homeomorphic to a connected sum of
sphere products, due to a result of McGavran
(cf.~\cite[Th.~6.3]{bo-me06}, see also Section~\ref{pentex} below).

\section{The case of a flag complex}\label{flagK}

A \emph{missing face} (or a \emph{minimal non-face}) of $K$ is a
subset $I\subset[m]$ such that $I\notin K$, but every proper
subset of $I$ is a simplex of~$K$. A simplicial complex $K$ is
called a \emph{flag complex} if each of its missing faces has two
vertices. Equivalently, $K$ is flag if any set of vertices of $K$
which are pairwise connected by edges spans a simplex.

In the case of flag complexes $K$ we shall show that the ``algebraic
approximations'' from the previous section are precise criterions
for the appropriate topological properties: $\zk$ is a wedge of
spheres precisely when $K$ is Golod, and $\zk$ is a connected
sum of sphere products if and only if $K$ is minimally
non-Golod.

There is the following description of $H_*(
\Omega\djs(K))=\Ext_{\k[K]}(\k,\k)$ for flag~$K$.

\begin{theorem}[{\cite[Th.~9.3]{pa-ra08}}]\label{flagldjs}
For any flag complex $K$, there is an isomorphism
\begin{equation}\label{flagalg}
  H_*\bigl( \Omega\djs(K);\k\bigr)\cong
  T\langle u_1,\ldots,u_m\rangle\big/
  \bigl(u_i^2=0,\; u_iu_j+u_ju_i=0\text{ for }\{i,j\}\in K\bigr)
\end{equation}
where $\k$ is a field and $T\langle u_1,\ldots,u_m\rangle$ is the
free associative algebra on $m$ generators of degree~1.
\end{theorem}

\begin{remark}
The theorem above is formulated in~\cite{pa-ra08} with
$\Q$-coefficients, but the argument (using the Adams cobar
construction and a result of Fr\"oberg~\cite{frob75} on
\emph{quadratic duality}) works for arbitrary field.

Algebra~\eqref{flagalg} may be viewed as a colimit (in the
category of noncommutative associative algebras) of a diagram of algebras over
the face category of~$K$, which assigns to each face $I\in K$ the
exterior algebra $\Lambda[u_i\colon i\in I]$. Another way to see
this algebra is to assign a generator $u_i$ satisfying $u_i^2=0$
to each vertex of~$K$, and think of each edge of $K$ as a
commutativity relation between the corresponding~$u_i$'s. The
resulting algebra is determined by the 1-skeleton (graph) of~$K$,
which is not surprising since~$K$ is flag. In the non-flag case
higher brackets appear, corresponding to higher Samelson products
in $ \Omega\djs(K)$, and the colimit above has to be replaced by a
homotopy colimit, see~\cite[\S8]{pa-ra08} for the details.

Algebra~\eqref{flagalg} is also known as the \emph{graph product
algebra} corresponding to the 1-skeleton of~$K$. Its
group-theoretic analogues are \emph{right-angled Artin} and
\emph{Coxeter groups}; in fact the polyhedral products of the form
$(\R P^\infty,\ast)^K$ and $(S^1,\ast)^K$ respectively are the
classifying spaces of these groups in the flag case
(cf.~\cite[\S4]{p-r-v04}).
\end{remark}

The \emph{$f$-vector} of $K$ is given by $\mb
f(K)=(f_0,\ldots,f_{n-1})$ where $f_i$ is the number of
$i$-dimensional faces and $n-1=\dim K$. The $h$-vector $\mb
h(K)=(h_0,h_1,\ldots,h_n)$ is defined from the relation
\[
  h_0t^n+h_1t^{n-1}+\cdots+h_n=(t-1)^n+f_0(t-1)^{n-1}+\cdots+f_{n-1}.
\]
The $h$-vector is symmetric for sphere triangulations~$K$; the
equations $h_i=h_{n-i}$ are known as the \emph{Dehn--Sommerville
relations}.

As another application of quadratic duality, the Poincar\'e series
of $H_*( \Omega\zk)$ can be calculated explicitly in terms of the
face numbers of~$K$ in the flag case.

\begin{proposition}[{\cite[Prop.~9.5]{pa-ra08}}]\label{poinser}
For any flag complex $K$, we have
\[
  P\bigl( H_*( \Omega\zk);t \bigr)\;=\;
  \frac1{(1+t)^{m-n}(1-h_1t+\cdots+(-1)^nh_nt^n)}\;.
\]
\end{proposition}

We now go further by identifying a minimal set of
multiplicative generators in $H_*( \Omega\zk)$ as a specific set
of iterated commutators of the~$u_i$.

\begin{theorem}\label{multgen}
Assume that $K$ is flag and $\k$ is a field. The algebra $H_*(
\Omega\zk;\k)$, viewed as the commutator
subalgebra~\eqref{flagalg} via exact sequence~\eqref{commutator},
is multiplicatively generated by $\sum_{I\subset[m]}\dim\widetilde
H^0(K_I)$ iterated commutators of the form
\[
  [u_j,u_i],\quad [u_{k_1},[u_j,u_i]],\quad\ldots,\quad
  [u_{k_1},[u_{k_2},\cdots[u_{k_{m-2}},[u_j,u_i]]\cdots]]
\]
where $k_1<k_2<\cdots<k_p<j>i$, $k_s\ne i$ for any~$s$, and $i$ is
the smallest vertex in a connected component not containing~$j$ of
the subcomplex $K_{\{k_1,\ldots,k_p,j,i\}}$. Furthermore, this
multiplicative generating set is minimal, that is, the commutators
above form a basis in the submodule of indecomposables
in~$H_*( \Omega\zk)$.
\end{theorem}

\begin{remark}
To help clarify the statement of Theorem~\ref{multgen}, it is useful
to consider which brackets $[u_{j},u_{i}]$ are in the list of multiplicative
generators for $H_{\ast}(\Omega\zk;\k)$. If $\{j,i\}\in K$ then
$i$ and $j$ are in the same connected component of the subcomplex
$K_{\{j,i\}}$, so $[u_{j},u_{i}]$ is not a multiplicative generator. On
the other hand, if $\{j,i\}\notin K$ then the subcomplex $K_{\{j,i\}}$
consists of the two distinct points $i$ and $j$, and $i$ is the smallest
vertex in its connected component of $K_{\{j,i\}}$ which does not
contain~$j$, so $[u_{j},u_{i}]$ is a multiplicative generator. In
Section~\ref{pentex} the example where $K$ is a pentagon is worked
out in detail, and in particular, a complete list of multiplicative
generators for $H_{\ast}(\Omega\zk;\k)$ is given.
\end{remark}

\begin{proof}
We observe that, for a given $I=\{k_1,\ldots,k_p,j,i\}$, the
number of the commutators containing all
$u_{k_1},\ldots,u_{k_p},u_j,u_i$ in the set above is equal to
$\dim\widetilde H^0(K_I)$ (one less the number of connected
components in $K_I$), so there are indeed
$\sum_{I\subset[m]}\dim\widetilde H^0(K_I)$ commutators in total.

We first prove a particular case of the statement, corresponding
to $K$ consisting of $m$ disjoint points. This result may be of
independent algebraic interest, as it is an analogue of the
description of a basis in the commutator subalgebra of a free
algebra, given by Cohen and Neisendorfer~\cite{co-ne84}.

\begin{lemma}
\label{CNzk}
Let $A$ be the commutator subalgebra of $T\langle
u_1,\ldots,u_m\rangle/(u_i^2=0)$, that is, the algebra defined by
the exact sequence
\[
  1\longrightarrow A\longrightarrow
  T\langle u_1,\ldots,u_m\rangle/
  (u_i^2=0)\longrightarrow\Lambda[u_1,\ldots,u_m]
  \longrightarrow1
\]
where $\deg u_i=1$. Then $A$ is a free associative algebra
minimally generated by the iterated commutators of the form
\[
  [u_j,u_i],\quad [u_{k_1},[u_j,u_i]],\quad\ldots,\quad
  [u_{k_1},[u_{k_2},\cdots[u_{k_{m-2}},[u_j,u_i]]\cdots]]
\]
where $k_1<k_2<\cdots<k_p<j>i$ and $k_s\ne i$ for any~$s$. Here,
the number of commutators of length $\ell$ is equal to
$(\ell-1)\bin m\ell$.
\end{lemma}
\begin{proof}
Let $S$ be the set of commutators in the statement of the lemma.
Let $B$ denote the commutator algebra of a free algebra on $m$
generators, that is, the algebra kernel of the map $T\langle
u_1,\ldots,u_m\rangle\to\Lambda[u_1,\ldots,u_m]$.
By~\cite{co-ne84}, $B$ is a free algebra generated the commutators
of the same form
$[u_{k_1},[u_{k_2},\cdots[u_{k_p},[u_j,u_i]]\cdots]]$, but with
the conditions $k_1<k_2<\cdots<k_p<j\ge i$ only. We therefore get
a larger set $T$ of commutators, in which $u_k$ may repeat.
However, note that the inequalities on the indices imply that if
$u_{k}$ repeats within a specified commutator, it does so only
once. We have $S\subseteq T$ and wish to show that any commutator
in $T-S$ is excluded from the multiplicative generating set of the
quotient $T\langle u_1,\ldots,u_m\rangle/(u_i^2=0)$. To see this,
induct on the length of the commutators, beginning with
$[u_{k},u_{k}]=2u_{k}^{2}=0$. Suppose the commutators of length~$<
n$ in $T$ have had any commutator with a repeating $u_{k}$
excluded from the generating set of $T\langle
u_1,\ldots,u_m\rangle/(u_i^2=0)$. Choose a commutator of length
$n$ with some $u_{k}$ repeating. Observe that it suffices to
consider commutators of the form
$[u_{k},[u_{k_2},\cdots[u_{k_p},[u_j,u_k]]\cdots]]$, which we
write as $[u_{k},[u_{k_{2}},c]]$ for $c=[u_{k_{3}},\cdots[
u_{k_{p}},[u_{j},u_{k}]]\cdots ]$. By the Jacobi identity,
$[u_{k},[u_{k_{2}},c]]=\pm [c,[u_{k},u_{k_{2}}]]\pm
[u_{k_{2}},[c,u_{k}]]$. Rewriting to conform to the restrictions
on the indices in the basis for $B$, we obtain
$[u_{k},[u_{k_{2}},c]]=\pm [c,[u_{k_{2}},u_{k}]]\pm
[u_{k_{2}},[u_{k},c]]$. The first term on the right is a
commutator of two elements of lower length in $S$. The second term
on the right has $[u_{k},c]$ excluded from the multiplicative
generating set of $T\langle u_1,\ldots,u_m\rangle/(u_i^2=0)$ by
inductive hypothesis, since $u_{k}$ appears in $c$. Therefore
$[u_{k},[u_{k_{2}},c]]$ is not a multiplicative generator of
$T\langle u_1,\ldots,u_m\rangle/(u_i^2=0)$.

Now observe that the set of commutators $S$ generates $A$
multiplicatively, since $A$ is a quotient of~$B/(u_i^2=0)$. To
show that $A$ is a free algebra, and the given generator set is
minimal, we use a topological argument. We have that $A=H_*(
\Omega\zk)$ where $K$ is a disjoint union of~$m$ points. By
Theorem~\ref{disjointpoints}, $\zk$ is homotopy equivalent to the
wedge of spheres $\bigvee_{\ell=2}^m(S^{\ell+1})^{\vee(\ell-1)\bin
m\ell}$. The Bott--Samelson Theorem implies that $A=H_*(
\Omega\zk)$ is a free algebra, and the number of generators in
each degree~$\ell$ agrees with the number of given commutators of
length~$\ell$.
\end{proof}

To complete the proof of Theorem~\ref{multgen} we must deal with
how the remaining relations in~(\ref{flagalg}), those of the form
$u_{i}u_{j}+u_{j}u_{i}=0$ if $\{i,j\}\in K$, affect the iterated
commutators listed in Lemma~\ref{CNzk}. Note that
$u_{i}u_{j}+u_{j}u_{i}=[u_{i},u_{j}]$ and that no~$u_{k}$ repeats
in any of the iterated commutators listed in Lemma~\ref{CNzk}.

Assume that $i,i^{\prime}$ are vertices in the same connected
component of $K$. Then there are vertices $i_{1}=i,
i_{2},\ldots,i_{k-1},i_{k}=i^{\prime}$ for some $k$ with the
property that the edges $\{i_{1},i_{2}\},\ldots,\{i_{k-1},i_{k}\}$
are all in~$K$. Arguing inductively as in the proof of
Lemma~\ref{CNzk}, the Jacobi identity implies that any iterated
commutator of length $l$ involving all $u_{i_1},\ldots,u_{i_k}$
can be rewritten as a sum of iterated commutators formed from
iterated commutators of lengths~$<l$. In particular, if $K$ is
connected (with $m$ vertices) then any iterated commutator of
length $m$ is zero modulo commutators of lesser length.

Continuing, suppose that we are given an index set
$I=\{k_{1},\ldots,k_{p},j,i\}$ with $k_{1}<k_{2}<\cdots
<k_{p}<j>i$ and $k_{s}\neq i$ for any $s$. Consider iterated
commutators of length $p+2$ involving one occurrence of $u_{k}$
for each $k\in I$. One example is
$[u_{k_1},[u_{k_2},\cdots[u_{k_{p}},[u_j,u_i]]\cdots]]$. Observe
that the restrictions on the order of the indices imply that the
only other examples occur by interchanging $u_{i}$ and $u_{k_{l}}$
provided $k_{l-1}<i<k_{l+1}$. Now if $i,j$ are in the same
connected component of $K_{I}$ then
$[u_{k_1},[u_{k_2},\cdots[u_{k_{p}},[u_j,u_i]]\cdots]]=0$ modulo
iterated commutators of lesser length, by the argument in the
previous paragraph applied to $K_I$. So to obtain nontrivial
commutators we require that $i,j$ appear in different components.
Also, if $\{k_{l_{1}},\ldots,k_{l_{r}}\}$ is the subset of
$\{k_{1},\ldots,k_{p}\}$ which lie in the same connected component
of $K_{I}$ as~$i$, then the iterated commutators
$[u_{k_1},[u_{k_2},\cdots,u_{k_{l_t-1}},[u_{i},[u_{k_{l_t+1}},
\cdots[u_{k_{p}},[u_j,u_{k_{l_{t}}}]]\cdots]]$ and
$[u_{k_1},[u_{k_2},\cdots[u_{k_{p}},[u_j,u_i]]\cdots]]$ can be
identified modulo iterated commutators of lesser lengths. So to
enumerate the one independent iterated commutator, we use the
convention of writing
$[u_{k_1},[u_{k_2},\cdots[u_{k_{p}},[u_j,u_i]]\cdots]]$ where $i$
is the smallest vertex in its connected component within $K_{I}$.
This leaves us with precisely the set of iterated commutators in
the statement of the theorem.

At this point, we have shown that the set of iterated commutators
in the statement of the theorem multiplicatively generates
$H_{\ast}( \Omega\zk)$. It remains to show that this is a minimal
generating set. To see this, it suffices to show that if
$I=\{k_{1},\ldots,k_{p},j,i\}$ where $k_{1}<\cdots <k_{p}<j>i$,
then the remaining iterated commutators on this index set are
algebraically independent. Let $\{k_{l_{1}},\ldots,k_{l_{r}}\}$ be
the subset of $\{k_{1},\ldots,k_{p}\}$ whose elements lie in the
same connected component of $K_I$ as~$i$. Let $K_{\widehat I}$ be
the full subcomplex of $K_{I}$ on the vertex set
$I-\{k_{l_{1}},\ldots,k_{l_{r}}\}$. There is a projection
$K_{I}\to K_{\widehat I}$. Observe that the connected component of
${K}_{\widehat I}$ containing the vertex $i$ is precisely the
singleton $\{i\}$, and there is a one-to-one correspondence
between the remaining iterated commutators of the form
$[u_{k_1},[u_{k_2},\cdots[u_{k_{p}},[u_j,u_i]]\cdots]]$ in
%Taras
$H_{\ast}(\Omega\djs(K_{I}))$
%$H_{\ast}(\djs(K_{I}))$
and the iterated commutators of length $(p+2)-r$ in $H_{\ast}(
\Omega \djs(K_{\widehat I}))$ formed by deleting the elements
$u_{k_{l}}$ whenever $k_{l}\in\{k_{l_{i}},\ldots,k_{l_{r}}\}$. The
latter set is algebraically independent since, topologically,
$\djs(K_{\widehat I})$ is the wedge $\mathbb{C}P^{\infty}\vee
\djs(K_{\widehat I}-\{i\})$, and the iterated commutators
correspond to independent Whitehead products in $\varSigma
\Omega\mathbb{C}P^{\infty}\wedge \Omega D\simeq
     \varSigma S^{1}\wedge \Omega D$,
where $D=\djs(K_{\widehat I}-\{i\})$. Hence the former set is
algebraically independent, as required.
\end{proof}

We now come to identifying the class of flag complexes $K$ for
which $\zk$ has homotopy type of a wedge of spheres.

Let $\Gamma$ be a graph on the vertex set~$[m]$. A \emph{clique}
of $\Gamma$ is a subset $I$ of vertices such that every two
vertices in $I$ are connected by an edge. Obviously, each flag
complex $K$ is the \emph{clique complex} of its one-skeleton
$\Gamma=K^1$, that is, the simplicial complex formed by filling in
each clique of $\Gamma$ by a face.

A graph $\Gamma$ is called \emph{chordal} if each of its cycles
with $\ge 4$ vertices has a chord (an edge joining two vertices
that are not adjacent in the cycle). Equivalently, a chordal graph
is a graph with no induced cycles of length more than three.

The following result gives an alternative characterisation of
chordal graphs.

\begin{theorem}[\cite{fu-gr65}]
A graph is chordal if and only if its vertices can be ordered in
such a way that, for each vertex~$i$, the lesser neighbours
of~$i$ form a clique.
\end{theorem}

Such an order of vertices is called a \emph{perfect elimination
ordering}.

\begin{theorem}\label{flws}
Let $K$ be a flag complex and $\k$ a field. The following
conditions are equivalent:
\begin{itemize}
\item[(a)] $\k[K]$ is a Golod ring;
\item[(b)] the multiplication in $H^*(\zk)$ is trivial;
\item[(c)] $\Gamma=K^1$ is a chordal graph;
\item[(d)] $\zk$ has homotopy type of a wedge of spheres.
\end{itemize}
\end{theorem}
\begin{proof}
(a)$\Rightarrow$(b) This is by definition of the Golod property
and Theorem~\ref{zkcoh}.

(b)$\Rightarrow$(c) Assume that $K^1$ is not chordal, and choose
an induced chordless cycle $I$ with $|I|\ge4$. Then the full
subcomplex $K_I$ is the same cycle (the boundary of an $|I|$-gon),
and therefore $\mathcal Z_{K_I}$ is a connected sum of sphere
products. Hence, $H^*(\mathcal Z_{K_I})$ has nontrivial products
(this can be also seen directly by using Theorem~\ref{zkcoh}).
Then, by Theorem~\ref{zkcoh}, the same nontrivial products appear
in~$H^*(\zk)$.

(c)$\Rightarrow$(d) Assume that the vertices of $K$ are in
%Taras
perfect
%total
elimination order. We assign to each vertex $i$ the clique $I_i$
consisting of $i$ and the lesser neighbours of~$i$. Each maximal
face of $K$ (that is, each maximal clique of~$K^1$) is obtained in
this way, so we get an induced order on the maximal faces:
$I_{i_1},\ldots,I_{i_s}$. Then, for each~$k=1,\ldots,s$, the
simplicial complex $\bigcup_{j<k}I_{i_j}$ is flag (since it is the
full subcomplex $K_{\{1,2,\ldots,i_{k-1}\}}$ in a flag complex).
The intersection $\bigr(\bigcup_{j<k}I_{i_j}\bigl)\cap I_{i_k}$ is
a clique, so it is a face of $\bigcup_{j<k}I_{i_j}$. Therefore,
$\zk$ has homotopy type of a wedge of spheres by
Corollary~\ref{orderws}.

(d)$\Rightarrow$(a) This is by definition of the Golod property
and the fact that the cohomology of the wedge of spheres contains
only trivial cup and Massey products.
\end{proof}

\begin{remark}
The equivalence of (a), (b) and (c) was proved
in~\cite[Th.~6.5]{be-jo07}.

All the implications in the above proof except (c)$\Rightarrow$(d)
are valid for arbitrary $K$, with the same arguments. However,
(c)$\Rightarrow$(d) fails in the non-flag case;
Example~\ref{rp2ex} is a counterexample.
\end{remark}

\begin{corollary}
Assume that $K$ is flag with $m$ vertices and $\zk$ has homotopy
type of a wedge of spheres. Then
%Taras
\begin{itemize}
\item[(a)]
the maximal dimension of spheres in the wedge is~$m+1$;
\item[(b)]
the number of spheres of dimension $\ell+1$ in the wedge is given
by $\sum_{|I|=\ell}\dim\widetilde H^0(K_I)$, for $2\le\ell\le m$;
\item[(c)] $H^i(K_I)=0$ for $i>0$ and all~$I$.
\end{itemize}
%Assume that $K$ is flag and $\zk$ has homotopy type of a wedge of
%spheres. Then the number of spheres of dimension $\ell+1$ in the
%wedge is given by $\sum_{|I|=\ell}\dim\widetilde H^0(K_I)$, for
%$2\le\ell\le m$. In particular, $H^i(K_I)=0$ for $i>0$ and
%all~$I$.
\end{corollary}
\begin{proof}
%Taras
If $\zk$ is a wedge of spheres, then $H_*(\Omega\zk)$ is a free
algebra on generators described by Theorem~\ref{multgen}, which
implies (a) and (b). It also follows that
$H^*(\zk)\cong\bigoplus_{I\subset[m]}\widetilde H^0(K_I)$. On the
other hand, $H^*(\zk)\cong\bigoplus_{I\subset[m]}\widetilde
H^*(K_I)$ by Theorem~\ref{zkcoh}, whence (c) follows.
%The first statement follows from Theorem~\ref{multgen}. The second
%one follows from Theorem~\ref{zkcoh}.
\end{proof}

\begin{theorem}\label{flcs}
Assume that $K$ is flag and $\k$ a field. The following
conditions are equivalent:
\begin{itemize}
\item[(a)] $K$ is minimally non-Golod;
\item[(b)] $\zk$ is homeomorphic to a connected sum of sphere products.
\end{itemize}
\end{theorem}
\begin{proof}
Indeed, if $K$ is flag and minimally non-Golod, then it is the
boundary of an $m$-gon with $m\ge4$.
\end{proof}

\section{The homotopy type of $ \Omega\zk$ when $K$ is flag}
\label{sec:loopzk}

In general, the homotopy type of $\zk$ when $K$ is a flag complex
may not be easy to determine. We have shown that $\zk$ has the
homotopy type of a wedge of spheres if $K$ is Golod, and $\zk$ has
the homotopy type of a connected sum of sphere products if $K$ is
minimally non-Golod. Beyond these two classes, it is not clear
what the homotopy type of $\zk$ may be. However, we will show in
Theorem~\ref{Whflag} that the homotopy type of $ \Omega\zk$
localised away from $2$ is a product of spheres and loops on
spheres.

To begin, suppose that $K$ is a flag complex on $m$ vertices.
Let $\overline K$ be the disjoint union of the $m$ vertices. Then
the inclusion
\[i\colon\namedright{\overline K}{}{K}\]
induces an inclusion
\[
  \djs(i)\colon\namedright{\djs(\overline K)=
  \bigvee_{j=1}^{m}\mathbb{C}P^{\infty}}{}{\djs(K)}
\]
and we obtain a homotopy pullback diagram
\begin{equation}
  \label{Kpb}
  \diagram
        \barzk\rto^-{\overline{f}}\dto^{\z(i)} & \djs(\overline K)\rto\dto^{D\!J(i)}
              & \prod_{i=1}^{m}\mathbb{C}P^{\infty}\ddouble \\
        \zk\rto^-{f} & \djs(K)\rto & \prod_{i=1}^{m}\mathbb{C}P^{\infty}
  \enddiagram
\end{equation}
which defines the maps $\z(i), \overline{f}$ and $f$.

It is useful to have some initial algebraic information.

\begin{lemma}
   \label{torsionfree}
   Let
   \(f\colon\namedright{X}{}{Y}\)
   be a map between two simply-connected spaces. If
   $H_{\ast}( \Omega X;\mathbb{Z})$ is torsion-free
   and $( \Omega f)_{\ast}$ is onto for coefficients in
   any field, then $H_{\ast}( \Omega Y;\mathbb{Z})$ is also torsion-free.
\end{lemma}

\begin{proof}
Suppose $H_{\ast}( \Omega Y;\mathbb{Z})$ is not torsion-free. Then
there is a prime $p$ and elements $b,\bar b\in H_{\ast}( \Omega
Y;\mathbb{Z}/p\mathbb{Z})$ such that $\beta^{r}\bar b=b$, where
$\beta^{r}$ is the $r^{th}$-Bockstein. As $( \Omega f)_{\ast}$ is
onto in mod-$p$ homology, there are elements
%Taras
$a,\bar a\in H_{\ast}( \Omega X;\mathbb{Z}/p\mathbb{Z})$
%$a,\bar a\in H_{\ast}( \Omega Y;\mathbb{Z}/p\mathbb{Z})$
such that
$( \Omega f)_{\ast}(a)=b$ and $( \Omega f)_{\ast}(\bar a)=\bar b$.
As $\beta^{r}$ commutes with $( \Omega f)_{\ast}$, we obtain
\[( \Omega f)_{\ast}(\beta^{r}\bar a)=\beta^{r} ( \Omega f)_{\ast}(\bar a)=
      \beta^{r}\bar b=b,\]
implying that $\beta^{r}\bar a\neq 0$.
This contradicts the fact that $H_{\ast}( \Omega X;\mathbb{Z})$
is torsion-free.
\end{proof}

\begin{corollary}
   \label{zktorsionfree}
   Let $K$ be a flag complex. Then $H_{\ast}( \Omega\zk;\mathbb{Z})$
   is torsion-free.
\end{corollary}

\begin{proof}
Observe that $ \Omega\djs(\overline K) \simeq T^{m}\times
\Omega\barzk$ and by Theorem~\ref{disjointpoints}, $\barzk$ is
homotopy equivalent to a wedge of spheres. Thus $H_{\ast}(
\Omega\djs(\overline K))$ is torsion-free. By
Theorem~\ref{flagldjs},
%and Lemma~\ref{CNzk},
%Taras
$(\Omega\djs(i))_{\ast}$
%$(\Omega DJ(i))_{\ast}$
is onto for coefficients in any field. So
by Lemma~\ref{torsionfree}, $H_{\ast}( \Omega\zk;\mathbb{Z})$ is
torsion-free.
\end{proof}

We now show that $ \Omega\zk$ for $K$ flag is homotopy
equivalent to a product of spheres and loops on spheres, when
localised rationally or at any prime $p\neq 2$.

\begin{theorem}
   \label{Whflag}
   Let $K$ be a flag complex. The following hold when localised rationally
   or at any prime $p\neq 2$:
   \begin{itemize}
      \item[(a)] the map
               \(\llnamedright{ \Omega\djs(\overline K)}{ \Omega D\!J(i)}{}{ \Omega\djs(K)}\)
               has a right homotopy inverse;
      \item[(b)] the map
               \(\lnamedright{ \Omega\barzk}{ \Omega\z(i)}{ \Omega\zk}\)
               has a right homotopy inverse;
     \item[(c)] $ \Omega\djs(K)$ and $ \Omega\zk$ are homotopy equivalent
               to products of spheres and loops on spheres.
   \end{itemize}
\end{theorem}

\begin{remark}
Theorem~\ref{Whflag} may be true integrally.
Corollary~\ref{zktorsionfree} says there are no obstructions
arising from torsion homology classes. When $K$ is Golod, so $\zk$
is homotopy equivalent to a wedge of spheres, then the integral
statement is a consequence of the Hilton--Milnor Theorem. When $K$
is minimally non-Golod, so $\zk$ is homeomorphic to a connected
sum of sphere products, then the integral statement holds
by~\cite{be-th12}. The methods in~\cite{be-th12} arise in a
different context and may or may not adapt to the case of $\zk$
for a general flag complex; at present not enough information is
known about $\zk$. The methods presented below may possibly be
fine tuned to prove the integral case, but more delicate
information would have to be known about the commutators in
$H_{\ast}( \Omega\zk)$. In particular, Theorem~\ref{multgen} gives
a minimal multiplicative basis for $H_{\ast}( \Omega\zk)$, but we
do not know enough about potential relations among them.
\end{remark}

\begin{proof}
We begin with an integral argument to establish some equivalences
between statements in the theorem.
After looping~(\ref{Kpb}), we obtain a homotopy pullback diagram
\[\diagram
         \Omega\barzk\rto^-{ \Omega\overline{f}}\dto^{ \Omega\z(i)}
              &  \Omega\djs(\overline K)\rto\dto^{ \Omega D\!J(i)}
              & T^{m}\ddouble \\
         \Omega\zk\rto^-{ \Omega f} &  \Omega\djs(K)\rto & T^{m}.
  \enddiagram\]
Since the fibration along the top row splits, it induces a
splitting of the fibration along the bottom row. Therefore, using
the loop structures in $ \Omega\djs(\overline K)$ and
$ \Omega\djs(K)$ to multiply, we obtain a homotopy commutative
diagram of homotopy equivalences
\[\diagram
        T^{m}\times \Omega\barzk\rto^-{\simeq}\dto^{1\times \Omega\z(i)}
              &  \Omega\djs(\overline K)\dto^{ \Omega D\!J(i)} \\
        T^{m}\times \Omega\zk\rto^-{\simeq} &  \Omega\djs(K).
  \enddiagram\]
Thus $ \Omega\djs(i)$ has a right homotopy inverse if and only if
$ \Omega\z(i)$ has a right homotopy inverse. Further, as $
\Omega\djs(K)\simeq T^{m}\times \Omega\zk$, we see that $
\Omega\djs(K)$ is homotopy equivalent to a product of spheres and
loops on spheres if and only if $ \Omega\zk$ is.

Now localise rationally or at a prime $p\neq 2$. It remains to
show that $ \Omega\djs(i)$ has a right homotopy inverse and $
\Omega\djs(K)$ is homotopy equivalent to a product of spheres and
loops on spheres. By Theorem~\ref{flagldjs}, there are
isomorphisms
\begin{align*}
   H_*\bigl( \Omega\djs(\overline K);\bf k\bigr) & \cong
     T\langle u_1,\ldots,u_m\rangle\big/\bigl(u_i^2=0\bigr) \\
   H_*\bigl( \Omega\djs(K);\bf k\bigr) & \cong
    T\langle u_1,\ldots,u_m\rangle\big/
      \bigl(u_i^2=0,\; u_iu_j+u_ju_i=0\text{ for }\{i,j\}\in K\bigr)
\end{align*}
where $\bf k$ is $\mathbb{Q}$ if we are localised rationally or
$\mathbb{Z}/p\mathbb{Z}$ if localised at $p$.
The free tensor algebra $T\langle u_{1},\ldots,u_{m}\rangle$ is isomorphic
to $UL\langle u_{1},\ldots,u_{m}\rangle$, the universal enveloping algebra
of the free Lie algebra on $u_{1},\ldots,u_{m}$. The relations in the two
tensor algebras above are induced from relations imposed on the underlying
free Lie algebra $L\langle u_{1},\ldots,u_{m}\rangle$. For
as $2$ is inverted in $\bf k$, the relation $u_{i}^{2}=0$ is equivalent to the
relation $[u_{i},u_{i}]=0$, and as each $u_{i}$ is of degree $1$, we have
$u_{i}u_{j}+u_{j}u_{i}=[u_{i},u_{j}]$. Thus there are isomorphisms
\begin{align*}
   & T\langle u_1,\ldots,u_m\rangle\big/\bigl(u_i^2=0\bigr)\cong
        U\bigl(L\langle u_{1},\ldots,u_{m}\rangle\big/\bigl([u_{i},u_{i}]=0\bigr)\bigr) \\
   & T\langle u_1,\ldots,u_m\rangle\big/\bigl(u_i^2=0,\,
       u_iu_j+u_ju_i=0\text{ for }\{i,j\}\in K\bigr)\\
   & \hspace{3cm}\cong U\bigl(L\langle u_{1},\ldots,u_{m}\rangle\big/\bigl([u_{i},u_{i}]=0,\,
       [u_i,u_j]=0\text{ for }\{i,j\}\in K\bigr)\bigr).
\end{align*}
To simplify notation, let
\begin{align*}
   & \overline L= L\langle u_{1},\ldots,u_{m}\rangle\big/\bigl([u_{i},u_{i}]=0\bigr) \\
   & L=L\langle u_{1},\ldots,u_{m}\rangle\big/\bigl([u_{i},u_{i}]=0,\,
       [u_i,u_j]=0\text{ for }\{i,j\}\in K\bigr).
\end{align*}

Observe as well that in passing from loop space homology to
universal enveloping algebras of Lie algebras, the map $( \Omega
\djs(i))_{\ast}$ is modelled by $U(\pi)$, where
\[\pi\colon\namedright{\overline L}{}{L}\]
is the quotient map of Lie algebras. As a map of $\bf k$-modules,
$\pi$ has a right inverse. Thus if $\widetilde{L}$ is the kernel of $\pi$,
then by~\cite{co-mo-ne79} there is an isomorphism of left $U\widetilde L$-modules
\[U\overline L\cong U\widetilde L\otimes UL.\]
Taking associated graded modules if necessary, by the
Poincar\'{e}--Birkhoff--Witt Theorem we obtain an isomorphism of
$\bf k$-modules
\[S(\overline L)\cong S(\widetilde L)\otimes S(L)\]
where $S(\,\,)$ is the free symmetric algebra functor.

In this case the Poincar\'{e}--Birkhoff--Witt Theorem has a
geometric realisation. Since $\overline K$ is a disjoint union of
points, by Theorem~\ref{disjointpoints}, there is an integral
homotopy equivalence $\barzk\simeq\bigvee_{\ell=2}^{m}
(S^{\ell+1})^{\vee (\ell-1)\binom{m}{\ell}}$. Therefore there are
integral homotopy equivalences
\[
   \Omega\djs(\overline K)\simeq T^{m}\times
   \Omega\barzk\simeq T^{m}\times  \Omega\Bigl
  (\bigvee_{\ell=2}^{m} (S^{\ell+1})^{\vee (\ell-1)
  \binom{m}{\ell}}\Bigr).
\]
The Hilton--Milnor Theorem gives an explicit decomposition of the
loops on a wedge of spheres as an infinite product of looped
spheres. In our case, we obtain an integral homotopy equivalence
\begin{equation}
  \label{HMequiv}
   \Omega\djs(\overline K)\simeq T^{m}\times
      \prod_{\alpha\in\mathcal{I}} \Omega S_{\alpha}
\end{equation}
for some index set $\mathcal{I}$, where each $S_{\alpha}$ is a sphere.

Take homology in~(\ref{HMequiv}) with $\bf k$ coefficients. We
have $H_{\ast}(T^{m})\cong\Lambda[u_{1},\ldots,u_{m}]$, where each
$u_{i}$ is of degree one. That is,
$H_{\ast}(T^{m})\cong\bigotimes_{i=1}^{m} S(u_{i})$. Next, if the
dimension of $S_{\alpha}$ is odd, say $S_{\alpha}=S^{2k+1}$, then
$H_{\ast}( \Omega S_{\alpha})\cong\k[u_{\alpha}]$, where $\vert
u_{\alpha}\vert=2k$, so $H_{\ast}( \Omega S_{\alpha})\cong
S(u_{\alpha})$. If the dimension of $S_{\alpha}$ is even, say
$S_{\alpha}=S^{2k}$ then the $\bf k$-local splitting $ \Omega
S^{2k}\simeq S^{2k-1}\times \Omega S^{4k-1}$ implies that
$H_{\ast}( \Omega
S_{\alpha})\cong\Lambda[u_{\alpha}]\otimes\k[v_{\alpha}]$, where
$\vert u_{\alpha}\vert=2k-1$ and $\vert v_{\alpha}\vert=4k-1$, so
$H_{\ast}( \Omega S_{\alpha})\cong S(u_{\alpha})\otimes
S(v_{\alpha})$. Putting all this together, (\ref{HMequiv})~implies
that there is a coalgebra isomorphism
\[
  H_{\ast}\bigl( \Omega\djs(\overline K);\k\bigr)\cong
  \bigotimes_{\alpha'\in\mathcal{I'}} S(u_{\alpha'})
\]
where the index set $\mathcal{I'}$ consists of $\{1,2,\ldots,m\}$,
every $\alpha\in\mathcal{I}$ where $S_{\alpha}$ is of odd
dimension, and two indices $\alpha_{2k-1},\alpha_{4k-1}$ for every
$\alpha\in\mathcal{I}$ where $S_{\alpha}$ is of dimension~$2k$.

We now have two descriptions of
$H_{\ast}( \Omega\djs(\overline K))$ as symmetric algebras, so
there is an isomorphism
\[S(\overline L)\cong\bigotimes_{\alpha'\in\mathcal{I'}} S(u_{\alpha'}).\]
On the other hand, there is a decomposition
$S(\overline L)\cong S(\widetilde L)\otimes S(L)$, so we can
choose a new index set $\mathcal{J}\subseteq\mathcal{I'}$ such that
the composite
\begin{equation}
  \label{hlgycomp}
  \bigotimes_{\beta\in\mathcal{J}} S(u_{\beta})\hookrightarrow
      \nameddright{\bigotimes_{\alpha'\in\mathcal{I'}} S(u_{\alpha'})}{\cong}
      {S(\overline L)}{\mathrm{proj}}{S(L)}
\end{equation}
is an isomorphism. Write
$\mathcal{J}=\mathcal{J}_{1}\sqcup\mathcal{J}_{2}$ where
$\mathcal{J}_{1}$ (respectively $\mathcal{J}_{2}$) consists of all
those $\beta\in\mathcal{J}$ with $\vert u_{\beta}\vert$ odd
(respectively even). Observe that~(\ref{hlgycomp}) is induced in
homology by the composite
\[
  \Bigl(\prod_{\beta\in\mathcal{J}_{1}} S_{\beta}\Bigr)\times
  \Bigl(\prod_{\beta\in\mathcal{J}_{2}} \Omega S_{\beta}\Bigr)
  \hookrightarrow
  \llnameddright{T^{m}\times\prod_{\alpha\in\mathcal{I}}
   \Omega S_{\alpha}}{\simeq}
  { \Omega\djs(\overline K)}{ \Omega DJ(i)}
  { \Omega\djs(K)}.
\]
The left map exists $\k$-locally, since there is a $\bf k$-local
decomposition $ \Omega S^{2k}\simeq S^{2k-1}\times \Omega
S^{4k-1}$. Thus if we take $\phi$ to be the composite of the left
and middle maps above, then $\phi$ has property that $
\Omega\djs(i)\circ\phi$ induces an isomorphism in $\k$-homology.
This completes the proof.
\end{proof}

\section{An example: The boundary of a pentagon}
\label{pentex}

In this section we consider an example which illustrates many of the
ideas and results of the paper. This is most fully discussed once
the algebra generators
%Taras
of $H_{\ast}( \Omega\djs(K))$
%of$H_{\ast}( \Omega\djs(K))$
are geometrically realised by Samelson products, so we begin with
a general lemma.

Let $K$ be a flag complex which is Golod. As in
Section~\ref{sec:loopzk}, let $\overline K$ be the disjoint union
of the $m$ vertices in $K$. We obtain an inclusion
\(i\colon\namedright{\overline K}{}{K}\) which induces an
inclusion \(\djs(i)\colon\namedright{\djs(\overline
K)=\bigvee_{j=1}^{m}\mathbb{C}P^{\infty}}
      {}{\djs(K)}\).
For $1\le i\le m$, let~$\overline\mu_{i}$ be the composite
\[
  \overline\mu_{i}\colon\namedddright{S^{2}}{}{\mathbb{C}P^{\infty}}{}
  {\bigvee_{j=1}^{m}\mathbb{C}P^{\infty}}{D\!J(i)}{\djs(K)}
\]
where the left map is the inclusion of the bottom cell and the middle map
is the inclusion of the $i^{th}$-wedge summand. Let
\[\mu_{i}\colon\namedright{S^{1}}{}{ \Omega\djs(K)}\]
be the adjoint of $\overline\mu_{i}$. Then in the description of
$H_{\ast}( \Omega\djs(K))$ in~(\ref{flagalg}), the Hurewicz
image of $\mu_{i}$ is the algebra generator $u_{i}$.

Since the Samelson product commutes with the Hurewicz
homomorphism, the Hurewicz image of any iterated Samelson product
of the $\mu_{i}$'s is the corresponding iterated commutator of the
$u_{i}$'s. As well, in the homotopy fibration
\(\nameddright{ \Omega\zk}{}{ \Omega\djs(K)}{}
% changed to T^m, as this notation was used in Section 2-4
%{\prod_{i=1}^{m} S^{1}}
{T^m}\), since $\pi_k(T^m)=0$ for $k>1$, any iterated Samelson
product of the $\mu_{i}$'s composes trivially into $T^m$ and so
lifts to $ \Omega\zk$.

Since we are regarding $H_{\ast}( \Omega\zk)$ as the commutator
subalgebra of $H_{\ast}( \Omega\djs(K))$ via exact
sequence~(\ref{commutator}), we can regard the lift to $\Omega\zk$
of any iterated Samelson product of the $\mu_{i}$'s as having
the same Hurewicz image. Therefore, the algebra generators
\[
  [u_j,u_i],\quad [u_{k_1},[u_j,u_i]],\quad\ldots,\quad
  [u_{k_1},[u_{k_2},\cdots[u_{k_{m-2}},[u_j,u_i]]\cdots]]
\]
of $H_{\ast}( \Omega\zk)$ in Theorem~\ref{multgen}, with
restrictions on the indices as stated in the theorem, are the
Hurewicz images of the lifts to $ \Omega\zk$ of the iterated
Samelson products
\begin{equation}
  \label{htpygen}
  [\mu_j,\mu_i],\quad [\mu_{k_1},[\mu_j,\mu_i]],\quad\ldots,\quad
  [\mu_{k_1},[\mu_{k_2},\cdots[\mu_{k_{m-2}},[\mu_j,\mu_i]]\cdots]].
\end{equation}

\begin{lemma}
   \label{Whgolod}
   Let $K$ be a flag complex and $\k$ a field. Suppose that $K$ is
   Golod, or equivalently by Theorem~\ref{flws}, that $\zk$ is
   homotopy equivalent to a wedge of spheres. Then each sphere in this
   wedge maps to $\djs(K)$ by an iterated Whitehead product of the
   maps $\overline\mu_{1},\ldots\overline\mu_{m}$.
\end{lemma}

\begin{proof}
Since $\zk$ is homotopy equivalent to a wedge of spheres,
$H_{\ast}( \Omega\zk)$ is a free associative algebra, where each
algebra generator of degree $d$ corresponds to a sphere of
dimension $d+1$ in the wedge decomposition of $\zk$. On the other
hand, a minimal generating set for $H_{\ast}( \Omega\zk)$ is given
by the iterated commutators in Theorem~\ref{multgen}, so each
iterated commutator listed in Theorem~\ref{multgen} of degree $d$
corresponds to a sphere of dimension $d+1$ in the wedge
decomposition of $\zk$. Applying the map \(\namedright{
\Omega\zk}{}{ \Omega\djs(K)}\), these iterated commutators are the
Hurewicz images of the iterated Samelson products
in~(\ref{htpygen}). Therefore, adjointing, the spheres in the
wedge decomposition of $\zk$ map to $\djs(K)$ by the iterated
Whitehead products
\[
  [\overline\mu_j,\overline\mu_i],\quad [\overline\mu_{k_1},
      [\overline\mu_j,\overline\mu_i]],\quad\ldots,\quad
  [\overline\mu_{k_1},[\overline\mu_{k_2},\cdots
  [\overline\mu_{k_{m-2}},[\overline\mu_j,\overline\mu_i]]\cdots]]
\]
with restrictions on the indices as in Theroem~\ref{multgen}.
\end{proof}

Now let $K$ be the boundary of pentagon, shown in Fig.~\ref{fbpent}.
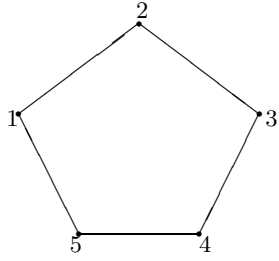
\begin{figure}[h]
\unitlength=0.8mm
  \begin{center}
  \begin{picture}(40,37)
  \put(8.5,-3){\small 5}
  \put(-2,18){\small 1}
  \put(19.5,36){\small 2}
  \put(41,18){\small 3}
  \put(30,-3){\small 4}
  \put(0,20){\circle*{1}}
  \put(20,35){\circle*{1}}
  \put(40,20){\circle*{1}}
  \put(30,0){\circle*{1}}
  \put(10,0){\circle*{1}}
  \put(10,0){\line(1,0){20}}
  \put(10,0){\line(-1,2){10}}
  \put(0,20){\line(4,3){20}}
  \put(20,35){\line(4,-3){20}}
  \put(40,20){\line(-1,-2){10}}
  \end{picture}
  \end{center}
  \caption{Boundary of pentagon.}
  \label{fbpent}
\end{figure}
Theorem~\ref{multgen} gives the
following 10 generators for the Pontryagin algebra
$H_*(\Omega\zk)$:
\begin{gather*}
  a_1=[u_3,u_1],\quad a_2=[u_4,u_1],\quad a_3=[u_4,u_2],\quad
  a_4=[u_5,u_2],\quad a_5=[u_5,u_3],\\
  b_1=[u_4,[u_5,u_2]],\quad b_2=[u_3,[u_5,u_2]],\quad
  b_3=[u_1 ,[u_5,u_3]],\\
  b_4=[u_3,[u_4,u_1]],\quad b_5=[u_2,[u_4,u_1]],
\end{gather*}
where $\deg a_i=2$ and $\deg b_i=3$. By Lemma~\ref{Whgolod},
$a_1$ is the Hurewicz image of the Samelson
product $[\mu_3,\mu_1]\colon S^2\to\Omega\djs(K)$ lifted to
$\Omega\zk$, and $b_1$ is the Hurewicz image of the iterated
Samelson product $[\mu_4,[\mu_5,\mu_2]]\colon S^3\to\Omega\djs(K)$
lifted to $\Omega\zk$; the other $a_i$ and $b_i$ are described
similarly. We therefore have adjoint maps
\[
  \epsilon\colon(S^2\vee
  S^3)^{\vee5}\to\Omega\zk\quad\text{and}\quad
  j\colon(S^3\vee S^4)^{\vee5}\to\zk
\]
corresponding to the wedge of all $a_i$ and~$b_i$. Now a
calculation using relations from Theorem~\ref{flagldjs} and the
Jacobi identity shows that $a_i$ and $b_i$ satisfy the relation
\begin{equation}\label{onerel}
  \sum_{i=1}^5 [a_i,b_i]=0
\end{equation}
(the signs can be made right by changing the order the elements
in the commutators defining $a_{i}$, $b_{i}$ if necessary).
This relation has a topological meaning. In general, suppose
that $M$ and $N$ are $d$-dimensional manifolds. Let
$\overline{M}$ be the $(d-1)$-skeleton of $M$, or equivalently,
$\overline{M}$ is obtained from $M$ by removing a disc in the
interior of the $d$-cell of $M$. Define $\overline{N}$ similarly.
Suppose that
\(f\colon\namedright{S^{d-1}}{}{\overline{M}}\)
and
\(g\colon\namedright{S^{d-1}}{}{\overline{N}}\)
are the attaching maps for the top cells in $M$ and $N$.
Then the attaching map for the top cell in the connected sum $M\cs N$ is
\(\namedright{S^{d-1}}{f+g}{\overline{M}\vee\overline{N}}\).
In our case, $S^{3}\times S^{4}$ is a manifold and the attaching map
\(\namedright{S^{6}}{}{S^{3}\vee S^{4}}\)
for its top cell is the Whitehead product $[s_{1},s_{2}]$, where $s_{1}$
and $s_{2}$ respectively are the inclusions of $S^{3}$ and $S^{4}$ into
$S^{3}\vee S^{4}$. The attaching map for the top cell of the $5$-fold connected
sum $(S^{3}\times S^{4})^{\cs 5}$ is therefore the sum of five
%Taras
such Whitehead products. Composing it with $j$ into~$\zk$ and
passing to the adjoint map we obtain
$\sum_{i=1}^{5}[a_{i},b_{i}]$.
%such Whitehead products, and composing it with $j$ into~$\zk$ we
%obtain $\sum_{i=1}^{5} [a_{i},b_{i}]$.
By~(\ref{onerel}), this sum is null homotopic. Thus the inclusion
$j\colon(S^3\vee S^4)^{\vee5}\to\zk$ extends to a map
\[
  \widetilde j\colon(S^3\times S^4)^{\cs5}\to\zk.
\]
Furthermore, a calculation using Theorem~\ref{zkcoh} shows that
$\widetilde j$ induces an isomorphism in cohomology
(see~\cite[Ex.~7.22]{bu-pa02}), that is, $\widetilde j$ is a
homotopy equivalence. Since both $(S^3\times S^4)^{\cs5}$ and
$\zk$ are manifolds, the complement of $(S^3\vee S^4)^{\vee5}$ in
$(S^3\times S^4)^{\cs5}$ and $\zk$ is a 7-disc, so that the
extension map $\widetilde j$ can be chosen to be one-to-one, which
implies that $\widetilde j$ is a homeomorphism.

We also obtain that $H_*(\Omega\zk)$ is the quotient of a free
algebra on ten generators $a_i,b_i$ by
relation~\eqref{onerel}. Its Poicar\'e series is given by
Proposition~\ref{poinser}:
\[
  P\bigl( H_*( \Omega\zk);t \bigr)\;=\;
  \frac1{1-5t^2-5t^3+t^5}\;.
\]
The summand $t^5$ in the denominator is what differs the
Poincar\'e series of the one-relator algebra $H_*(\Omega\zk)$ from
that of the free algebra $H_*(\Omega(S^3\vee S^4)^{\vee5})$.

A similar argument can be used to show that $\zk$ is homeomorphic
to a connected sum of sphere products when $K$ is a boundary of a
$m$-gon with $m\ge4$.

\end{document}